\documentclass[12pt,a4paper]{amsart}
\usepackage{amsfonts}
\usepackage[top=35mm, bottom=35mm, left=30mm, right=30mm]{geometry}
\usepackage[colorlinks=true,citecolor=blue]{hyperref}
\usepackage{mathptmx}
\usepackage{enumerate}
\usepackage{eucal}
\usepackage{graphicx}
\usepackage{amssymb}
\usepackage{amsmath}
\usepackage{amsthm}
\usepackage{color}

\usepackage{xcolor}
\newtheorem{thm}{Theorem}[section]
\newtheorem{cor}[thm]{Corollary}
\newtheorem{lem}[thm]{Lemma}

\theoremstyle{definition}

\newtheorem{rem}[thm]{Remark}
\numberwithin{equation}{section}

\newcommand{\N}{\mathbb{N}}

\newcommand{\Z}{\mathbb{Z}_+}

\newcommand{\eps}{\varepsilon}

\newcommand{\orb}{\mathrm{Orb}}

\DeclareMathOperator{\diam}{diam}

\begin{document}
\title{Devaney chaos plus shadowing implies distributional chaos}
\author[Jian Li]{Jian Li}
\address[Jian Li]{Department of Mathematics, Shantou University, Shantou 515063, Guangdong, P.R. China}
\email{lijian09@mail.ustc.edu.cn}
\author[Jie Li]{Jie Li}
\address[Jie Li]{Academy of Mathematics and Systems Science, Chinese Academy of Sciences, Beijing, 100190, P.R. China; and School of Mathematical Sciences, University of Science and Technology of China, Hefei, Anhui, 230026, P.R. China}
\email{jiel0516@mail.ustc.edu.cn}
\author[Siming Tu]{Siming Tu}
\address[Siming Tu]{Centro de Modelamiento Matematico, Universidad de Chile, Av. Beauchef 851, Piso 7, Santiago de Chile, Chile}
\email{tsiming@dim.uchile.cl}

\begin{abstract}
We explore connections among the regional proximal relation,
the asymptotic relation and the distal relation for a topological dynamical system
with the shadowing property,
and show that if a Devaney chaotic system has the shadowing property then it is distributionally chaotic.
\end{abstract}

\maketitle
\date{\today}

\textbf{Chaos theory is a hot topic in area of topological dynamics, and different definitions of chaos have been introduced.
In this paper, we consider the relationship between two important notions of chaos. One is Devaney chaos, which means that the system
is topologically transitive, sensitive to initial conditions and the set of periodic points in the space is dense. Another is called distributional
chaos, which is related to the complexity of the trajectory behavior of points in the space. The notion of Devaney chaos
is a global property of the dynamical system, while distributional chaos is a local one. Neither Li-Yorke chaos nor distributional
chaos can imply Devaney chaos in general. So a natural question is that under what kinds of conditions can one relate these two types of
chaos. This paper gives one such condition that under the shadowing property one can imply distributional chaos from Devaney chaos,
building a connection between the global and local chaotic behaviors of such systems.}

\section{Introduction}
Throughout this article, by a topological dynamical system we mean a pair $(X, T)$,
where $X$ is a compact metric space with a metric $d$ and $T\colon X\to X$ is a continuous map.
Let $\Z$ and $\mathbb{N}$ be the sets of nonnegative integers and natural numbers, respectively.

The study of chaos theory in the topological dynamics has attracted a lot of attentions.
Unfortunately, there is still no definitive, widely accepted mathematical definition of chaos.
Due to different understanding of the complex behaviors of dynamical systems,
various different but closely related definitions of chaos are introduced,
such as Li-Yorke chaos, Devaney chaos, distributional chaos and so on.
It is important to understand the relationships among the various definitions of chaos.
We refer the reader to the survey \cite{Li2016}
for recent developments of this topic.

Following the ideas in \cite{Li1975}, we usually define the Li-Yorke chaos as follows.
A dynamical system $(X,T)$ is called \emph{Li-Yorke chaotic} if there is an uncountable
subset $S$ of $X$ such that for every two distinct points $x,y\in S$, one has
\[
\liminf_{n\to\infty}\ d(T^nx,T^ny)=0,\text{ and }\limsup_{n\to\infty}\ d(T^nx,T^ny)>0.
\]
In \cite{Li1975}, Li and Yorke showed that if a continuous interval map
$f\colon [0,1]\to[0,1]$ has a periodic point of period $3$, then it is Li-Yorke chaotic.
Moreover, if an interval map has positive topological entropy then it is Li-Yorke chaotic,
and this is in fact true for general systems, as shown by Blanchard, Glasner, Kolyada, and
Maass \cite{BGKM}, and in fact one can know from their proof that if $(X,T)$ has an ergodic invariant
measure which is not measurable distal then same conclusion holds. However, there exist some interval
maps which are Li-Yorke chaotic while have zero topological entropy \cite{Xiong1986,Smital1986}.
Interested readers can find in the survey \cite{Li2016} more information on Li-Yorke chaos and
recent developments related to it.

In~\cite{Schweizer1994}, using ideas from probability theory, Schweizer and Sm{\'{\i}}tal introduced
a strengthening of Li-Yorke chaos,  so called distributional chaos.
For $x,y$ in $X$ and $n\in\N$, we define a function
$\Phi^n_{xy}$ on the real line by
\[ \Phi^n_{xy}(t)=\tfrac 1 n \#\bigl(\{0\leq i \leq n-1\colon d(T^ix, T^iy)<t\}\bigr),\]
where $\#(\cdot)$ denotes the number of elements of a finite set.
It is clear that $\Phi^n_{xy}$ is a distribution function.
A dynamical system is called \emph{distributionally chaotic}
if there is an uncountable subset $S$ of $X$ such that
for every non-diagonal pair $(x,y)\in S\times S$ is ditributionally scrambled, i.e., one has
\begin{enumerate}
\item for every $t>0$, $\limsup_{n\to\infty} \Phi^n_{xy}(t)=1$, and
\item there exists some $\delta>0$ such that  $\liminf_{n\to\infty} \Phi^n_{xy}(\delta)=0$.
\end{enumerate}
Since then it has evolved into three variants of so-called distributional chaos DC1, DC2 and DC3 (ordered from strongest to weakest),
and the distributional chaos is the DC1. Pikula \cite{P2007} showed that positive topological entropy does not imply DC1 chaos.
Sm\'ital conjectured that it does imply DC2 (see e.g. \cite{S2}). This problem has been around for several years and several
partial solutions had been given (see e.g. \cite{Oprocha2012} for systems with uniformly positive entropy or \cite{S2} for some other cases) before it was finally solved by Downarowicz \cite{Downarowicz2014} (see \cite{HLY14} for an alternative proof). Note that while in general it is not true that positive topological entropy does not imply
distibutional chaos, it is true for interval maps. In fact, it turns out that  a continuous map $f\colon [0,1]\to[0,1]$
has positive topological entropy if and only if it is distributionally chaotic \cite{Schweizer1994}.

Another popular definition of chaos was introduced by Devaney in~\cite{Devaney1989}.
A dynamical system $(X, T)$ is called \emph{Devaney chaotic} if it satisfies the following
three properties:
\begin{enumerate}
  \item $(X,T)$ is topological transitive;
  \item $(X, T)$ has sensitive dependence on initial conditions;
  \item the set of periodic points of $(X, T)$ is dense in $X$.
\end{enumerate}
Note that the condition of sensitivity is redundant in the definition of Devaney chaos~\cite{BBCDS92}.

It is clear that Devaney chaos is a global property of the dynamical system,
while Li-Yorke chaos or distributional chaos is a local one. Neither Li-Yorke chaos nor
distributional chaos can imply Devaney chaos in general.
In 2002, Huang and Ye showed that Devaney chaos implies Li-Yorke chaos \cite{Huang2002}.
But Devaney chaos does not always imply distributional chaos,
as there are some examples of Devaney chaotic systems
without distributionally scrambled pairs \cite{Oprocha2006},
see  also  \cite{Falniowski2015} for another approach.
But in some special case, Devaney chaos can imply distributional chaos.
It is shown in~\cite{Li1993} that an interval map has positive topological
entropy if and only if it has a Devaney chaotic subsystem,
then also if and only if it is distributionally chaotic.
For a subshift of finite type it is
distributionally chaotic if and only if it has a Devaney chaotic subsystem~\cite{Wang2005}.
It is well known that a subshift is of finite type if and only
if it has the shadowing property~\cite{Walters1978}.
Recall that a sequence $\{x_n\}_{n=0}^\infty$ is a \emph{$\delta$-pseudo-orbit} for $T$
if $d(x_{{n+1}},Tx_{n})<\delta$ for all $n\in\Z$, and
is \emph{$\epsilon$-traced} by a point $x\in X$ if $d(T^n x,x_n)<\epsilon$ for all $n\in\Z$.
We say that a dynamical system $(X,T)$ has the \emph{shadowing property}
if for any $\eps>0$ we can find a $\delta>0$ such that
each $\delta$-pseudo-orbit for $T$ is $\eps$-traced by some point of $X$.

In \cite{Arai2007} Arai and Chinen introduced the concept of $P$-chaos
by changing the condition of transitivity in the definition of Devaney chaos to the shadowing property,
and they proved that when the state space $X$ is connected if a dynamical system $(X,T)$ is $P$-chaotic,
then it is Devaney chaotic with positive topological entropy and
exhibits a distributionally scrambled pair.
Note that in zero-dimensional case $P$-chaotic system may have zero topological entropy,
see~\cite[Example 4.5]{Arai2007}.
We do not know whether $P$-chaos with positive topological entropy can always imply distributional chaos.

Another natural question inspired by~\cite{Arai2007} is
whether Devaney chaos under the shadowing property can imply distributional chaos.
In this note we shall give a positive answer to this question.
In fact, we will prove a bit more. More specifically, we show that
it implies the generalized multivariant distributional chaos~\cite{Tan2009}.
Recall that for a dynamical system $(X,T)$ and $n\geq 2$,
an $n$-tuple $(x_1,x_2,\dotsc,x_n)\in X^n$ is \emph{distributionally $n$-scrambled}
if it satisfies
\begin{enumerate}
\item for any $t>0$, the upper density of $\{n\in\Z\colon d(T^nx_i,T^nx_j)<t,\ 1\leq i<j\leq n\}$
is one, and
\item there exists some $\delta>0$ such that
the upper density of $\{n\in\Z\colon d(T^nx_i,T^nx_j)>\delta,\ 1\leq i<j\leq n \}$ is one.
\end{enumerate}
A subset $C$ of $X$ is called \emph{distributionally $n$-scrambled} if
any pairwise distinct $n$ points in $C$ form a distributionally $n$-scrambled tuple.
A dynamical system $(X,T)$ is called \emph{distributionally $n$-chaotic} if
there is an uncountable distributionally $n$-scrambled set in $X$.

Similarly, if the separated constant $\delta$ is uniform for all pairwise distinct $n$-tuples in $C$,
we can define \emph{distributionally $n$-$\delta$-scrambled sets}
and \emph{distributional $n$-$\delta$-chaos}.
Note that the notions of distributional $n$-chaos and distributional $(n+1)$-chaos
are different in general \cite{Tan2014,Li2013O}.

Now we are ready to state our main result as follows.
\begin{thm}\label{thm:Devany}
Let $(X,T)$ be a non-periodic transitive system and has the shadowing property. Then
\begin{enumerate}
\item\label{thm:Devaney:1} if $(X,T)$ has a fixed point, then there exists a dense Mycielski subset $K$ of $X$ such that
$K$ is distributionally $n$-$\delta_n$-scrambled for all $n\geq 2$ and some $\delta_n>0$.
\item \label{thm:Devaney:2} if $(X,T)$ has a periodic point, then  there exists a Mycielski subset $K$ of $X$ such that
$K$ is distributionally $n$-$\delta_n$-scrambled for all $n\geq 2$ and some $\delta_n>0$.
\end{enumerate}
In particular, a Devaney chaotic system with the shadowing property is distributionally chaotic.
\end{thm}

This paper is organized as follows. For preparation we recall in Section \ref{sect:preliminaries}
some basic definitions and results that will be needed.
In Section \ref{sect:main-proof} we introduce the multivariant version of
the notions of regional proximity, $\eps$-asymptoticity (for given $\eps>0$) and distality,
and explore their connections for dynamical system with the shadowing property,
see Lemmas~\ref{lem:asy-eps} and~\ref{lem:distal-eps}.
As an application we yield the proof of Theorem~\ref{thm:Devany}.

\section{Preliminaries}\label{sect:preliminaries}
In this section we introduce some basic notions and facts in topological dynamics, which will be used later.

Let $F$ be a subset of $\Z$. The \emph{upper density} of $F$ is defined by
\[\overline{d}(F)=\limsup_{n\to\infty}\frac 1 n \#(F\cap\{0,1,\dots,n-1\}).\]

Let $X$ be a compact metric space. A subset $A$ of $X$ is called
a \emph{Cantor set} if it is homeomorphic to  the standard Cantor ternary set and
a \emph{Mycielski set} if it can be expressed as a union of countably many Cantor sets.
For convenience we restate here a
version of Mycielski's theorem \cite[Theorem~1]{Mycielski1964} which we shall use.
\begin{thm}[Mycielski Theorem]\label{thm:Mycielski-thm}
Let $X$ be a perfect compact metric space.
If for every $n\geq 2$, $R_n$ is a dense $G_\delta$ subset of $X^n$, then there exists a
dense Mycielski subset $K$ of $X$ such that for any distinct $n$ points $x_1,x_2,\dotsc,x_n\in K$,
the tuple $(x_1,x_2,\dotsc,x_n)$ is in $R_n$.
\end{thm}

Let $(X,T)$ be a dynamical system.
If a closed subset $Y$ of $X$ is $T$-invariant, i.e., $T(Y)\subset Y$,
then the restriction $(Y, T|_Y)$ of $(X,T)$ to $Y$ is itself a dynamical system,
and we call it a \emph{subsystem} of $(X,T)$. If there is no ambiguous,
we will denote the restriction $T|_Y$ by $T$ for simplicity.

A dynamical system $(X,T)$ is called \emph{transitive} if for every two non-empty open subsets $U$ and $V$
of $X$ there is an $n\in \N$ such that $T^n U\cap V\neq\emptyset$. A point $x\in X$ is called a
\emph{transitive point} if the orbit of $x$, $\orb(x,T)=\{x,Tx,T^2x,\dotsc\}$, is in dense in $X$ i.e.,
$\overline{\orb(x,T)}=X$. In our setting, when $X$ is a
compact metric space, if $(X,T)$ is transitive then
the collection of transitive points is a dense $G_\delta$ set in $X$.
If a transitive system has an isolated point then the system consists of just one periodic orbit.
So for a non-periodic transitive system $(X,T)$, the space $X$ is always perfect.

A dynamical system $(X,T)$ is called \emph{weakly mixing} if the product system $(X\times X,T\times T)$ is transitive.
For any $n\in\N$, the $n$-fold product system of $(X,T)$ is denoted by $(X^n, T^{(n)})$.
It is not hard to see that for any $n\in\N$, $(X^n, T^{(n)})$ is weakly mixing provided that $(X,T)$ is.

Consider $\{0,1\}$ as a topology space with the discrete topology
and let $\Sigma=\{0,1\}^{\Z}$ with the product topology.
Then $\Sigma$ is a Cantor space.
We write the elements of $\Sigma$ in the form as $x=x_0x_1x_2\dotsc$.
The \emph{shift map} from $\sigma$ to itself is defined by
$\sigma(x)_n=x_{n+1}$ for $n\in\Z$. It is clear that $\sigma$ is a continuous surjection.
The dynamical system $(\Sigma,\sigma)$ is called the \emph{full shift}.
If $X$ is non-empty, closed and $\sigma$-invariant
% (i.e. $\sigma(X)\subset X$),
then we call the dynamical system $(X, \sigma)$ a \emph{subshift}.

For an open cover $\mathcal{U}$ of $X$, let $N(\mathcal{U})$ denote the smallest cardinality
of a subcover of $\mathcal{U}$. By the compactness of $X$,
$N(\mathcal{U})$ is always finite.
If $\mathcal{U}$ and $\mathcal{V}$ are two open covers of $X$,
then $\mathcal{U}\bigvee \mathcal{V}=\{U\cap V:U\in \mathcal{U}, V\in \mathcal{V}\}$
is called their common refinement.
Let $\mathcal{U}_n=\mathcal{U}\bigvee  T^{-1}\mathcal{U}\bigvee \dotsb\bigvee  T^{-n+1}\mathcal{U}$,
where $T^{-k}\mathcal{U}=\{T^{-k}U: U\in \mathcal{U}\}$.
The \emph{topological entropy of $\mathcal{U}$ with respect to $T$} is defined by
\[h_{\textrm{top}}(\mathcal{U},T)=\lim_{n\to+\infty}\frac{\log N(\mathcal{U}_n)}{n},\]
and the \emph{topological entropy of $(X,T)$} is
\[h_{\textrm{top}}(T)=\sup h_{\textrm{top}}(\mathcal{U},T), \]
where the supremum ranges over all open covers $\mathcal{U}$ of $X$.
We refer to the textbook \cite{Wal82} for more details related to the topological entropy.

A non-diagonal pair $(x_1 ,x_2 )\in X^2$ is said to be an \emph{entropy pair}
if whenever $W_1$ and $W_2$ are closed disjoint
neighbourhoods of $x_1$ and $x_2$,
the open cover $\{W_1^c , W_2^c \}$ has positive topological entropy \cite{Blanchard1992}.
More generally, we call a non-diagonal tuple $\mathbf{x} = (x_1 , \dots, x_n)\in X^n$ an
\emph{entropy tuple} if whenever $W_1, \dotsc , W_l$ are closed pairwise
disjoint neighbourhoods of the distinct points in the list
$x_1 , \dots , x_n$, the open cover
$\{W_1^c , \dots , W_l^c \}$ has positive topological entropy \cite{Huang2006}.
We say that a dynamical system $(X,T)$ has \emph{uniformly positive entropy of all orders}
if for each $n\geq 2$, every non-diagonal tuple in $X^n$ is an entropy tuple.

Let $(X, T)$ and $(Y,S)$ be two dynamical systems.
We say that a continuous map $\pi\colon X\to Y$  is a \emph{factor map}
if $\pi$ is onto and intertwines the actions, i.e., $\pi\circ T=S\circ\pi$,
and in which case $(Y,S)$ is called a \emph{factor} of $(X, T)$.
The following theorem says that the entropy tuples can be lifted from a factor.

\begin{thm}[\cite{Huang2006}, Proposition 2.4]\label{thm:lift}
Let $\pi\colon (X,T)\to (Y,S)$ be a factor map. If $(y_1,\ldots,y_n)$ is an entropy $n$-tuple in $Y$ then
there exist $x_1,\ldots,x_n\in X$ such that $\pi(x_i)=y_i$ and $(x_1,\ldots,x_n)$ is an entropy
$n$-tuple in $X$.
\end{thm}

A dynamical system $(X,T)$ is said to have \emph{sensitive dependence on initial conditions}
(or just \emph{sensitive}) if there exists some $\delta>0$ such that for each $x\in X$ and each $\eps>0$
there are $y\in X$ with $d(x,y)<\eps$ and $k\in\N$ such that $d(T^kx,T^ky)>\delta$.
In \cite{Xiong2005} Xiong generalized the notion of sensitivity to multi-variant sensitivity.
Using ideas from the local entropy theory, Ye and Zhang~\cite{Ye2008}
introduced the notion of sensitive tuples.
An $n$-tuple $(x_1,x_2,\dotsc,x_n)\in X^n$ is called \emph{sensitive}
if $x_i\neq x_j$ for all $1\leq i<j\leq n$ and
for any $n\geq 2$, any neighborhood $U_i$ of $x_i$, $i=1,2,\dotsc,n$, and any
non-empty open subset $U$ of $X$
there exist $k\in\N$ and $y_i\in U$ such that $T^k(y_i)\in U_i$ for $i=1,2,\dotsc,n$.
We will use the follow result.
\begin{thm}[\cite{Ye2008}, Theorem 4.4]\label{thm:entr-tup-Sen-tup}
If a dynamical system $(X,T)$ is transitive, then for every $n\geq 2$,
any entropy $n$-tuple is sensitive.
\end{thm}

\section{Proof of the main result}\label{sect:main-proof}
In this section, we will prove the main result Theorem~\ref{thm:Devany}.
First we need to introduce some notions.

For $n\geq 2$, an $n$-tuple $(x_1,x_2,\dotsc,x_n)\in X^n$ is called \emph{regionally proximal}
if for every $\eps>0$ there exist $y_1,y_2,\dotsc,y_n\in X$ and $k\in\N$
such that $d(x_i,y_i)<\eps$ for $i=1,2,\dotsc,n$ and
$\diam(\{T^ky_1,T^ky_2,\dotsc,T^ky_n\})<\eps$.
For a given $\eps>0$, an $n$-tuple $(x_1,x_2,\dotsc,x_n)\in X^n$ is called
\emph{$\eps$-asymptotic} if
\[\limsup_{k\to\infty} \ \diam(\{T^kx_1,T^kx_2,\dotsc,T^kx_n\})<\eps.\]

\begin{rem}
The regionally proximal relation plays an important role in the study of topological dynamics
One of the important properties of this relation is that
it is a closed invariant equivalence relation for any minimal system and
the quotient space obtained from this relation is the maximal equicontinuous factor of the original system.
We refer the interested readers to \cite{Auslander1988} and \cite{EE2014} for details on this topic.
\end{rem}

Denote by $Q_n(X,T)$ and $Asy_n^\eps(X,T)$ the collection of all regionally proximal $n$-tuples
and $\eps$-asymptotic $n$-tuples respectively.
It is clear that $Q_n(X,T)$ is always closed and $\bigcap_{\eps>0}Asy_n^\eps(X,T) \subset Q_n(X,T)$.
For a dynamical system with the shadowing property,
we have the following relationship.
\begin{lem} \label{lem:asy-eps}
If a dynamical system $(X,T)$ has the shadowing property, then
for every $n\geq 2$ and $\eps>0$, $Q_n(X,T)\subset \overline{Asy_n^\eps(X,T)}$.
\end{lem}
\begin{proof}
Fix $(x_1,x_2,\dotsc,x_n)\in Q_n(X,T)$ and $\eps>0$.
For every $0<\eta<\eps/2$, pick $0<\delta<\eta$ provided that every $\delta$-pseudo-orbit is $\eta$-traced.

Since $(x_1,x_2,\dotsc,x_n)\in Q_n(X,T)$, there exist $y_1,y_2,\dotsc,y_n\in X$ and
$k\in\N$ such that
\[
    d(x_i,y_i)<\delta,\ i=1,2,\dotsc,n
\]
and
\[
    \diam (\{T^{k_1}y_1,T^{k_1}y_2,\dotsc,T^{k_1}y_n\})<\delta.
\]
For every $i=1,2,\dotsc,n$, the sequence
\[
    y_i,Ty_i,\dotsc,T^{k-1}y_i,T^ky_1,T^{k+1}y_1,\dotsc
\]
is a $\delta$-pseudo-orbit.
Thus by the shadowing property for each $i=1,\ldots,n$, there exists $z_i\in X$ such that
\begin{align*}
d(T^jz_i,T^jy_i)&<\eta, \ j=0,1,\ldots,k-1, \text{ and }\\
d(T^jz_i,T^jy_1)&<\eta, \ j=k,k+1,\ldots,
\end{align*}
which implies that $(z_1,z_2,\dotsc,z_n)$ is $\eps$-asymptotic.
For each $i=1,\ldots,n$,
\[
    d(x_i,z_i)<d(x_i,y_i)+d(y_i,z_i)<\delta+\eta<2\eta,
\]
which ends the proof.
\end{proof}

For a given $\eps>0$, an $n$-tuple $(x_1,x_2,\dotsc,x_n)\in X^n$ is called
\emph{$\eps$-distal} if
\[\liminf_{k\to\infty} \min_{1\leq i<j\leq n} d(T^kx_i,T^kx_j)>\eps.\]
We say that an $n$-tuple is \emph{distal} if it is $\eps$-distal for some $\eps>0$.
Denote by $D_n^\eps(X,T)$ the collection of all $\eps$-distal $n$-tuples.

\begin{lem} \label{lem:distal-eps}
Let $(X,T)$ be a transitive system with the shadowing property.
If $(X,T)$ is sensitive, then for every $n\geq 2$ there exists $\eps>0$
such that  $Q_n(X,T)\subset \overline{D_n^\eps(X,T)}$.
\end{lem}
\begin{proof}
We first show that for every $n\geq 2$, there exists an $n$-tuple which is both sensitive and distal.
Since $(X,T)$ is sensitive, by \cite[Proposition 3.1]{Li2013}
there exists an $m\in\N$, a subsystem $(Y,T^m)$ of $(X,T^m)$ and a factor map $\pi$ from $(Y,T^m)$
to a full shift $(\Sigma,\sigma)$.
Since the set of periodic points of $(\Sigma,\sigma)$ is dense,
pick $n$ pairwise distinct periodic points $u_1,u_2,\dotsc,u_n$ in $\Sigma$.
Then $(u_1,u_2,\dotsc,u_n)$ is distal.
As $(\Sigma,\sigma)$ has uniformly positive entropy of all orders,
$(u_1,u_2,\dotsc,u_n)$  is also an entropy $n$-tuple.
By Theorem \ref{thm:lift} there exists an entropy $n$-tuple $(v_1,v_2,\dotsc,v_n)\in Y^n$ such that
$\pi (v_i)=u_i$ for $i=1,2,\dotsc,n$.
Moreover, $(v_1,v_2,\dotsc,v_n)$ is also distal in $(Y,T^m)$, as its image under $\pi$ is distal.
Then we have that $(v_1,v_2,\dotsc,v_n)$ is an entropy and distal $n$-tuple for $(X,T)$,
because so is it for $(Y,T^m)$.
Now by~Theorem \ref{thm:entr-tup-Sen-tup}, $(v_1,v_2,\dotsc,v_n)$ is  sensitive.

From the definition of distal tuple, we can pick an $\eps>0$ such that
\[\liminf_{k\to\infty} \min_{1\leq i<j\leq n} d(T^kv_i,T^kv_j)>2\eps.\]
For every $0<\eta<\eps/2$, we pick $0<\delta<\eta$ provided that every $2\delta$-pseudo-orbit is $\eta$-traced.
Now fix $(x_1,x_2,\dotsc,x_n)\in Q_n(X,T)$. There exist $y_1,y_2,\dotsc,y_n\in X$ and $k_1\in\N$
such that
\[ d(x_i,y_i)<\delta,\ i=1,2,\dotsc,n\]
and
\[\diam (\{T^{k_1}y_1,T^{k_1}y_2,\dotsc,T^{k_1}y_n\})<\delta.\]
As $(v_1,v_2,\dotsc,v_n)$ is a sensitive tuple, there exist $z_1,z_2,\dotsc,z_n\in X$ and $k_2\in\N$
such that
\[
    d(z_i, T^{k_1}y_1)<\delta
\]
and
\[
    d(T^{k_2}z_i,v_i)<\delta, i=1,2,\dotsc,n.
\]
From the above construction it is not hard to see that for every $i=1,2,\dotsc,n$, the sequence
\[
    y_i,Ty_i,\dotsc,T^{k_1-1}y_i,z_i,Tz_i,\dotsc T^{k_2-1}z_i,v_i,Tv_i,\dotsc
\]
is a $2\delta$-pseudo-orbit.
Hence there exists $w_i\in X$ such that
\begin{align*}
&d(T^jw_i,T^jy_i)<\eta,\ j=0,1,\dotsc,k_1-1,\\
&d(T^jw_i,T^{j-k_1}z_i)<\eta,\ j=k_1,k_1+1,\dotsc,k_1+k_2-1, \text{ and }\\
& d(T^jw_i,T^{j-k_1-k_2}v_i)<\eta,\ j=k_1+k_2,k_1+k_2+1,\dotsc.
\end{align*}
Then
\[
\liminf_{k\to\infty} \min_{1\leq i<j\leq n} d(T^kw_i,T^kw_j)>
\liminf_{k\to\infty} \min_{1\leq i<j\leq n} d(T^kv_i,T^kv_j)-2\eta>2\eps-2\eta>\eps.
\]
Moreover, for each $i=1,2,\dotsc,n$,
\[
d(w_i,x_i)<d(w_i,y_i)+d(y_i,x_i)<\delta+\eta<2\eta,
\]
which ends the proof.
\end{proof}

\begin{thm}\label{thm:distScra}
Let $(X,T)$ be a non-trivial transitive system with the shadowing property.
If $Q_n(X,T)=X^n$ for some $n\ge 2$, then there exists a dense
Mycielski subset $K$ of $X$ such that
$K$ is distributionally $n$-$\delta_n$-scrambled for  some $\delta_n>0$.
\end{thm}
\begin{proof}
%Let
%$$
%R_n=\{(x_1,x_2,\ldots,x_n)\in X^n\colon \overline{d}(\{n\in\Z\colon d(T^nx_i,T^nx_j)<t,\ 1\leq i<j\leq n\})=1, \forall t>0 \}
%$$
%and
%$$
%S_n(\delta)=\{(x_1,x_2,\ldots,x_n)\in X^n\colon \overline{d}(\{n\in\Z\colon d(T^nx_i,T^nx_j)>\delta,\ 1\leq i<j\leq n \})=1 \}.
%$$
Let $R_n$ be the collection of all $n$-tuples $(x_1,x_2,\ldots,x_n)$ satisfying
for any $t>0$ the upper density of $\{k\in\Z\colon d(T^kx_i,T^kx_j)<t,\ 1\leq i<j\leq n\}$ is one,
and $S_n(\delta)$ be the collection of all $n$-tuples satisfying
the upper density of $\{k\in\Z\colon d(T^kx_i,T^kx_j)>\delta,\ 1\leq i<j\leq n \}$ is one.
It is not hard to see that
\begin{align*}
R_n=&\
\bigcap_{\ell=1}^\infty\bigcap_{q=1}^\infty\bigcap_{p=1}^\infty\bigcup_{m=p}^\infty
\bigg\{(x_1,x_2,\ldots,x_n)\in X^n\colon \\
&\frac{\#\big(\big\{k\in \{0,1,\ldots,m-1\}\colon d(T^kx_i,T^kx_j)<\frac{1}{\ell},\ 1\leq i<j\leq n\big\}\big)}{m}>1-\frac{1}{q}\bigg\}
\end{align*}
and
\begin{align*}
S_n(\delta)=&\ \bigcap_{q=1}^\infty\bigcap_{p=1}^\infty\bigcup_{m=p}^\infty
\bigg\{(x_1,x_2,\ldots,x_n)\in X^n\colon \\
&\frac{\#\big(\big\{k\in \{0,1,\ldots,m-1\}\colon d(T^kx_i,T^kx_j)>\delta,\ 1\leq i<j\leq n\big\}\big)}{m}>1-\frac{1}{q}\bigg\}.
\end{align*}
Then $R_n$ and $S_n(\delta)$ are $G_{\delta}$ subsets of $X^n$.

Now assume $Q_n(X,T)=X^n$ for some $n\ge 2$. It is not hard to see that  $\bigcap_{\eps>0} Asy_n^{\eps}(X,T)\subset R_n$.
This, together with Lemma~\ref{lem:asy-eps} and $Q_n(X,T)=X^n$,
immediately implies that $R_n$ is dense in $X^n$.
Since $(X,T)$ is transitive and has the shadowing property,
it is either sensitive or equicontinuous,
see for instance~\cite[Theorem 6]{Moothathu2011} or~\cite[Theorem 3.3]{Li2013}.
But $Q_n(X,T)=X^n$ then $(X,T)$ can not be equicontinuous and hence it is sensitive.
By Lemma~\ref{lem:distal-eps} there is a $\delta_n>0$ such that
$Q_n(X,T)\subset \overline{D_n^{\delta_n}(X,T)}$.
Observe that  $D_n^\eps(X,T)\subset S_n(\eps)$ for each $\eps>0$ and $Q_n(X,T)=X^n$,
then $S_n(\delta_n)$ is directly dense in $X^n$.

Note that the collection of distributionally $n$-$\delta_n$-scrambled tuples
is just the intersection of $R_n$ and $S_n(\delta)$,
so it is also a dense $G_{\delta}$ subset of $X^n$ by the Baire category theorem.
As $X$ is perfect, the result then follows from the Mycielski Theorem.
\end{proof}

Recall that a dynamical system $(X,T)$ is weakly mixing
if the product system $(X\times X,T\times T)$ is transitive.
It is well known that if $(X,T)$ is weakly mixing, then the $n$-th product system $(X^n,T^{(n)})$
is also transitive, where $T^{(n)}=X\times X\times\dotsb \times X$ ($n$-times).
For a weakly mixing system $(X,T)$,
it is not hard to see that $Q_n(X,T)=X^n$ for all $n\geq 2$.
So We have the following corollary.

\begin{cor}\label{cor:weak-mixing-shadowing}
If a non-trivial weakly mixing system $(X,T)$ has shadowing property,
then there exists a dense
Mycielski subset $K$ of $X$ such that
$K$ is distributionally $n$-$\delta_n$-scrambled for all $n\geq 2$ and some $\delta_n>0$.
\end{cor}

It is shown in~\cite{Huang2002} that a transitive system with a fixed point
contains a dense uncountable scrambled set
(by Mycielski Theorem \ref{thm:Mycielski-thm} this uncountable scramble set
can be chosen to be a Mycielski set).
Furthermore, if the system is additionally sensitive,
then it contains a dense Mycielski $\delta$-scrambled set for some $\delta>0$ \cite{Mai2004}.
In \cite{Xiong2005} Xiong generalized these results to $n$-scrambled sets for any $n\geq 2$.
Our main result Theorem \ref{thm:Devany} states that under the shadowing property,
these results can be improved to distributional $n$-scrambled sets.
Now we are going to prove the main result.

\begin{proof}[Proof of Theorem \ref{thm:Devany}]
\eqref{thm:Devaney:1}
Assume $(X,T)$ has a fixed point $x_0$. Since $(X,T)$ is transitive,
there exists a transitive point $x\in X$ and an increasing sequence $\{n_i\}$ such that $T^{n_i}x\to x_0$.
As $x_0$ is a fixed point, we have $T^{n_i+j}x\to x_0$ for any fixed $j\in\Z$.
This implies that each $n$-tuple $(T^{i_1} x, \ldots, T^{i_n} x)$ is regionally proximal
for all $i_1,\ldots,i_n\in\Z$.
Since $x$ is a transitive point, the orbit of $x$ is dense.
So the collection
\[\{(T^{i_1} x, \ldots, T^{i_n} x)\colon i_1,\ldots,i_n\in\Z\}\]
is dense in $X^n$ and then $Q_n(X,T)=X^n$ for all $n\geq 2$.
Now the result follows from~Theorem~\ref {thm:distScra}.

\eqref{thm:Devaney:2}
Assume that $(X,T)$ has a periodic point of period $p$.
Let $x$ be a transitive point and $Y_i=\overline{\orb(T^ix,T^p)}$ for $i=0,1,\dotsc,p-1$.
There exists $q\in \N$ with $q|p$ such that we have a decomposition
\[
    X=Y_0\cup Y_1\cup \dots\cup Y_{q-1},
\]
where $Y_i\neq Y_j$ for $0\leq i<j<q$ and $T(Y_i)=Y_{i+1\pmod{q}}$.
Moreover, the interior of $Y_0$ (in $X$) is dense in $Y_0$.
Now consider the system $(Y_0,T^p)$, it is transitive and has a fixed point.
By~\eqref{thm:Devaney:1}, we know that for each $n\geq 2$, $Q_n(Y_0,T^p)=Y_0^n$.
Then $Y_0^n\subset Q_n(X,T^p)\subset Q_n(X,T)$.
Following the proof of Theorem \ref{thm:distScra}
we know that the
set of distributionally $n$-$\delta_n$-scrambled tuples of $X^n$ is dense $G_\delta$ in $Y_0^n$.
As $Y_0$ is perfect, applying the Mycielski Theorem we can get a dense Mycielski subset $K$ of $Y_0$
which is distributionally $n$-$\delta_n$-scrambled.
\end{proof}

\begin{rem}\label{rem:Devaney-DC}
In fact, from the proof of Theorems~\ref{thm:Devany} and~\ref{thm:distScra},
we know that when $(X,T)$ is transitive with shadowing property, if there exists a closed subset $Y$ of $X$ with non-empty interior
such that $Y\times Y\subset Q_2(X,T)$ then $(X,T)$ is  distributionally chaotic.
It should be noticed that there are some distributionally chaotic systems
which do not have this property.
For example, for the product system of the full shift and an odometer,
it is transitive and has the shadowing property.
It is distributionally chaotic and has positive entropy, because so is the full shift.
But the interior of the regional proximal relation is empty.
This leads us to ask the following question.
\end{rem}

\noindent\textbf{Question:}
For a transitive system with the shadowing property, if it has the positive topological entropy, is it distributionally chaotic (in the sense of DC1)?

\medskip

As said in the introduction, there exist examples which are of positive topological entropy but not DC1.
However, since we know that positive topological entropy implies chaos DC2 \cite{Downarowicz2014, HLY14},
and so hopefully with some additional conditions
we can get DC1 from positive topological entropy, and the shadowing property maybe is a suitable one.
Note that at this time the system does not need to have a periodic point,
for example the product system of an odometer with the tent map.
This means that more tools are needed to be developed for the investigation of the above question.

\medskip

\noindent\textbf{Final Remarks.}
One of the referees suggested that
there is another approach to get a weaker form of Theorem~\ref{thm:Devany}.
By Theorem 6 in \cite{RW2008}, a transitive system with fixed point is chain mixing.
It is easy to see that a chain mixing map with the shadowing property is topologically mixing.
By Theorem 1 in \cite{KO12},  we know that if a transitive system has the shadowing property and
a fixed point, then it has the specification property.
Now applying Theorem 1 in \cite{Oprocha2007} it is distributionally chaotic.

Note that our proof of Theorem~\ref{thm:Devany} is new and
Lemmas \ref{lem:asy-eps} and \ref{lem:distal-eps} are of independent interest.

\section{Acknowledgments}
The authors would like to thank the anonymous referees for
his/her valuable comments and suggestions, especially for reminding
several earlier results can give a weaker form of Theorem~\ref{thm:Devany},
as mentioned in the Final Remarks.

Jian Li was supported in part by NNSF of China (11401362 and 11471125), Jie Li was supported in part
by NNSF of China (11371339, 11431012 and 11571335) and Siming Tu was supported by FONDECYT Grant 3150002.

\end{document}